\documentclass{article}
\usepackage{amsmath, amsthm}
\usepackage{amssymb}
\newtheoremstyle{theorem}
  {12pt}          
  {12pt}      
  {\sl}  
  {\parindent}     
  {\bf}  
  {. }    
  { }    
  {}     
\newtheorem{thm}{Theorem}[section]

\newtheorem{lem}[thm]{Lemma}
\newtheorem{prop}[thm]{Proposition}
\newtheorem{defn}{Definition}[section]

\newtheorem{rem}{Remark}[section]
\begin{document}

\thispagestyle{empty} \setcounter{page}{1}



\begin{center}
{\Large\bf Monotonicity Results for Delta and Nabla Caputo and Riemann Fractional Differences via dual identities}

\vskip.20in

Thabet Abdeljawad $^{a}$\footnote{} \; and \; Bahaaeldin Abdalla $^{a}$\footnote{}
\\[2mm]

{\footnotesize $^a$Department of Mathematics and Physical
Sciences,  Prince Sultan University\\
{\footnotesize P. O. Box 66833, Riyadh 11586, Saudi Arabia}\\

tabdeljawad@psu.edu.sa, babdallah@psu.edu.sa}

\end{center}

\vskip.2in

{\footnotesize { \textbf{Abstract}: Recently, some authors have proved monotonicity results for delta and nabla fractional differences separately. In this article, we use dual identities relating delta and nabla fractional difference operators to prove shortly the monotonicity properties for the (left Riemann) nabla fractional differences using the corresponding delta type properties. Also, we proved some monotonicity properties for  the Caputo fractional differences. Finally, we use the $Q-$operator dual identities to prove monotonicity results for the right fractional difference operators.  }
 \\


\section{Introduction and preliminaries about fractional sums and differences}
Fractional calculus have attracted many researchers in different fields of engineering and science since not short time \cite{podlubny, Samko, Kilbas}. The extent of interest in this field reaches every concept can be applied to fractional dynamical systems such as delay , impulse, stability, controllability, biological modelling, variational calculus, etc. \cite{Sad, Thabet,FTH,Higher var, Cont, machado,thabetetal}. Discrete fractional calculus remained without serious developing til the beginning of the last decade in the last century. Twenty years after the articles \cite{Gray, Miller},  many authors started to attack discrete fractional calculus very extensively (\cite{Th Caputo}-\cite{Holm}). Recently, some authors started to study monotonicity and convexity properties of delta and nabla (left Riemann) fractional differences (\cite{Dahal}-\cite{Christ}). For example, the authors studied the monotonicity properties for delta fractional differences of order $0<\alpha <1$ while others studied the case $1<\alpha <2$. In \cite{Jia} the authors proved two monotonicity results for delta and nabla fractional differences separately (see Theorem A and Theorem B there). Then, very recently, the authors in \cite{Slov} improved the results obtained for the delta case by using better starting conditions. In this article, we use the dual identities relating delta and nabla (left Riemann) fractional differences \cite{Thsh, Thdiscrete} to provide monotonicity short proofs for the nabla case using the delta case. Then, we used the relation between Riemann and Caputo fractional differences to carry the analysis in (\cite{Dahal}-\cite{Christ}) from Riemann fractional differences to Caputo fractional differences. Finally, we used the action of the $Q-$operator \cite{Th Caputo, Thsh, Thdiscrete} in relating left and right fractional difference types to prove monotonicity results for right fractional difference types.

For a natural number $n$, the fractional polynomial is defined by,

 \begin{equation} \label{fp}
 t^{(n)}=\prod_{j=0}^{n-1} (t-j)=\frac{\Gamma(t+1)}{\Gamma(t+1-n)},
 \end{equation}
where $\Gamma$ denotes the special gamma function and the product is
zero when $t+1-j=0$ for some $j$. More generally, for arbitrary
$\alpha$, define
\begin{equation} \label{fpg}
t^{(\alpha)}=\frac{\Gamma(t+1)}{\Gamma(t+1-\alpha)},
\end{equation}
where the convention that division at pole yields zero.
Given that the forward and backward difference operators are defined
by
\begin{equation} \label{fb}
\Delta f(t)=f(t+1)-f(t)\texttt{,}~\nabla f(t)=f(t)-f(t-1)
\end{equation}
respectively, we define iteratively the operators
$\Delta^m=\Delta(\Delta^{m-1})$ and $\nabla^m=\nabla(\nabla^{m-1})$,
where $m$ is a natural number.

 Here are some properties of the  factorial function.

\begin{lem} \label{pfp} (\cite{Ferd1})
Assume the following factorial functions are well defined.

(i) $ \Delta t^{(\alpha)}=\alpha  t^{(\alpha-1)}$.

(ii) $(t-\mu)t^{(\mu)}= t^{(\mu+1)}$, where $\mu \in \mathbb{R}$.

(iii) $\mu^{(\mu)}=\Gamma (\mu+1)$.

(iv) If $t\leq r$, then $t^{(\alpha)}\leq r^{(\alpha)}$ for any
$\alpha>r$.

(v) If $0<\alpha<1$, then $ t^{(\alpha\nu)}\geq
(t^{(\nu)})^\alpha$.

(vi) $t^{(\alpha+\beta)}= (t-\beta)^{(\alpha)} t^{(\beta)}$.
\end{lem}

Also, for our purposes we list down the following two properties, the proofs of which are straightforward.

\begin{equation} \label{ou1}
\nabla_s (s-t)^{(\alpha-1)}=(\alpha-1)(\rho(s)-t)^{(\alpha-2)}.
\end{equation}

\begin{equation} \label{ou2}
\nabla_t
(\rho(s)-t)^{(\alpha-1)}=-(\alpha-1)(\rho(s)-t)^{(\alpha-2)}.
\end{equation}

For the sake of the nabla fractional calculus we have the following definition

\begin{defn} \label{rising}(\cite{Boros,Grah,Spanier})

(i) For a natural number $m$, the $m$ rising (ascending) factorial of $t$ is defined by

\begin{equation}\label{rising 1}
    t^{\overline{m}}= \prod_{k=0}^{m-1}(t+k),~~~t^{\overline{0}}=1.
\end{equation}

(ii) For any real number the $\alpha$ rising function is defined by
\begin{equation}\label{alpharising}
 t^{\overline{\alpha}}=\frac{\Gamma(t+\alpha)}{\Gamma(t)},~~~t \in \mathbb{R}-~\{...,-2,-1,0\},~~0^{\overline{\mathbb{\alpha}}}=0
\end{equation}

\end{defn}

Regarding the rising factorial function we observe the following:

(i) \begin{equation}\label{oper}
    \nabla (t^{\overline{\alpha}})=\alpha t^{\overline{\alpha-1}}
\end{equation}

 (ii)
 \begin{equation}\label{oper2}
    (t^{\overline{\alpha}})=(t+\alpha-1)^{(\alpha)}.
 \end{equation}

(iii)
\begin{equation}\label{oper3}
   \Delta_t (s-\rho(t))^{\overline{\alpha}}= -\alpha  (s-\rho(t))^{\overline{\alpha-1}}
\end{equation}

\textbf{Notation}:
\begin{enumerate}
\item[$(i)$] For a real $\alpha>0$, we set $n=[\alpha]+1$, where $[\alpha]$ is the greatest integer less than $\alpha$.

\item[$(ii)$] For real numbers $a$ and $b$, we denote $\mathbb{N}_a=\{a,a+1,...\}$ and $~_{b}\mathbb{N}=\{b,b-1,...\}$.

\item[$(iii)$]For $n \in \mathbb{N}$ , we denote
$$ _{\circleddash}\Delta^n f(t)\triangleq (-1)^n\Delta^n f(t).$$
\item[$(iv)$]For $n \in \mathbb{N}$, we denote
                   $$ \nabla_{\circleddash}^n f(t)\triangleq (-1)^n\nabla^n f(t).$$
\end{enumerate}

The following definition and the properties followed can be found in \cite{Thsh} and  the references therein.

\begin{defn} \label{fractional sums}\cite{Thsh}
Let $\sigma(t)=t+1$ and $\rho(t)=t-1$ be the forward and backward jumping operators, respectively. Then

(i) The (delta) left fractional sum of order $\alpha>0$ (starting from $a$) is defined by:
\begin{equation}\label{dls}
    \Delta_a^{-\alpha} f(t)=\frac{1}{\Gamma(\alpha)} \sum_{s=a}^{t-\alpha}(t-\sigma(s))^{(\alpha-1)}f(s),~~t \in \mathbb{N}_{a+\alpha}.
\end{equation}

(ii) The (delta) right fractional sum of order $\alpha>0$ (ending at  $b$) is defined by:
\begin{equation}\label{drs}
   ~_{b}\Delta^{-\alpha} f(t)=\frac{1}{\Gamma(\alpha)} \sum_{s=t+\alpha}^{b}(s-\sigma(t))^{(\alpha-1)}f(s)=\frac{1}{\Gamma(\alpha)} \sum_{s=t+\alpha}^{b}(\rho(s)-t)^{(\alpha-1)}f(s),~~t \in ~_{b-\alpha}\mathbb{N}.
\end{equation}

(iii) The (nabla) left fractional sum of order $\alpha>0$ (starting from $a$) is defined by:
\begin{equation}\label{nlf}
  \nabla_a^{-\alpha} f(t)=\frac{1}{\Gamma(\alpha)} \sum_{s=a+1}^t(t-\rho(s))^{\overline{\alpha-1}}f(s),~~t \in \mathbb{N}_{a+1}.
\end{equation}

(iv)The (nabla) right fractional sum of order $\alpha>0$ (ending at $b$) is defined by:
\begin{equation}\label{nrs}
   ~_{b}\nabla^{-\alpha} f(t)=\frac{1}{\Gamma(\alpha)} \sum_{s=t}^{b-1}(s-\rho(t))^{\overline{\alpha-1}}f(s)=\frac{1}{\Gamma(\alpha)} \sum_{s=t}^{b-1}(\sigma(s)-t)^{\overline{\alpha-1}}f(s),~~t \in ~_{b-1}\mathbb{N}.
\end{equation}
\end{defn}

Regarding the delta left fractional sum we observe the following:

(i) $\Delta_a^{-\alpha}$ maps functions defined on $\mathbb{N}_a$ to
functions defined on $\mathbb{N}_{a+\alpha}$.

(ii) $u(t)=\Delta_a^{-n}f(t)$, $n \in \mathbb{N}$, satisfies the
initial value problem
\begin{equation} \label{ivpf}
\Delta^n u(t)=f(t),~~t\in N_a,~u(a+j-1)=0,~ j=1,2,...,n.
\end{equation}

(iii) The Cauchy function $\frac{(t-\sigma(s))^{(n-1)}}{(n-1)!}$
vanishes at $s=t-(n-1),...,t-1$.

\indent

Regarding the delta right fractional sum we observe the following:

(i)  $~_{b}\Delta^{-\alpha}$ maps functions defined on $_{b}\mathbb{N}$ to
functions defined on $_{b-\alpha}\mathbb{N}$.

(ii) $u(t)=~_{b}\Delta^{-n}f(t)$, $n \in \mathbb{N}$, satisfies the
initial value problem
\begin{equation} \label{ivpb}
\nabla_\ominus^n u(t)=f(t),~~t\in~ _{b}N,~u(b-j+1)=0,~ j=1,2,...,n.
\end{equation}

(iii) the Cauchy function $\frac{(\rho(s)-t)^{(n-1)}}{(n-1)!}$
vanishes at $s=t+1,t+2,...,t+(n-1)$.

\indent

Regarding the nabla left fractional sum we observe the following:

(i) $ \nabla_a^{-\alpha}$ maps functions defined on $\mathbb{N}_a$ to functions defined on $\mathbb{N}_{a}$.

(ii)$ \nabla_a^{-n}f(t)$ satisfies the n-th order discrete initial value problem

\begin{equation}\label{s1}
    \nabla^n y(t)=f(t),~~~\nabla^i y(a)=0,~~i=0,1,...,n-1
\end{equation}

(iii) The Cauchy function $\frac{(t-\rho(s))^{\overline{n-1}}}{\Gamma(n)}$ satisfies $\nabla^n y(t)=0$.

\indent

Regarding the nabla right fractional sum we observe the following:

(i) $ ~_{b}\nabla^{-\alpha}$ maps functions defined on $~_{b}\mathbb{N}$ to functions defined on $~_{b}\mathbb{N}$.

(ii)$ ~_{b}\nabla^{-n}f(t)$ satisfies the n-th order discrete initial value problem

\begin{equation}\label{s2}
    ~_{\ominus}\Delta^n y(t)=f(t),~~~ ~_{\ominus}\Delta^i y(b)=0,~~i=0,1,...,n-1.
\end{equation}
The proof can be done inductively. Namely, assuming it is true for $n$, we have
\begin{equation}\label{t1}
    ~_{\ominus}\Delta^{n+1} ~_{b}\nabla^{-(n+1)}f(t)=~_{\ominus}\Delta^{n}[-\Delta ~_{b}\nabla^{-(n+1)}f(t)].
\end{equation}

By the help of (\ref{oper3}), it follows that
\begin{equation}\label{t2}
  ~_{\ominus}\Delta^{n+1} ~_{b}\nabla^{-(n+1)}f(t)=  ~_{\ominus}\Delta^{n} ~_{b}\nabla^{-n}f(t)=f(t).
\end{equation}
The other part is clear by using the convention that $\sum_{k=s+1}^s=0$.

(iii) The Cauchy function $\frac{(s-\rho(t))^{\overline{n-1}}}{\Gamma(n)}$ satisfies $_{\ominus}\Delta^n y(t)=0$.

\indent

\begin{defn} \label{fractional differences}
(i)\cite{Miller}  The (delta) left fractional difference of order $\alpha>0$ (starting from $a$ ) is defined by:
\begin{equation}\label{dls}
    \Delta_a^{\alpha} f(t)=\Delta^n \Delta_a^{-(n-\alpha)} f(t)= \frac{\Delta^n}{\Gamma(n-\alpha)} \sum_{s=a}^{t-(n-\alpha)}(t-\sigma(s))^{(n-\alpha-1)}f(s),~~t \in \mathbb{N}_{a+(n-\alpha)}
\end{equation}

(ii) \cite{TDbyparts} The (delta) right fractional difference of order $\alpha>0$ (ending at  $b$ ) is defined by:
\begin{equation}\label{drd}
   ~_{b}\Delta^{\alpha} f(t)=  \nabla_{\circleddash}^n  ~_{b}\Delta^{-(n-\alpha)}f(t)=\frac{(-1)^n \nabla ^n}{\Gamma(n-\alpha)} \sum_{s=t+(n-\alpha)}^{b}(s-\sigma(t))^{(n-\alpha-1)}f(s),~~t \in ~_{b-(n-\alpha)}\mathbb{N}
\end{equation}

(iii) \cite{Holm} The (nabla) left fractional difference of order $\alpha>0$ (starting from $a$ ) is defined by:
\begin{equation}\label{nld}
  \nabla_a^{\alpha} f(t)=\nabla^n \nabla_a^{-(n-\alpha)}f(t)= \frac{\nabla^n}{\Gamma(n-\alpha)} \sum_{s=a+1}^t(t-\rho(s))^{\overline{n-\alpha-1}}f(s),~~t \in \mathbb{N}_{a+1}
\end{equation}

(iv) (\cite{THFer}, \cite{Thsh} The (nabla) right fractional difference of order $\alpha>0$ (ending at $b$ ) is defined by:
\begin{equation}\label{nrd}
   ~_{b}\nabla^{\alpha} f(t)= ~_{\circleddash}\Delta^n ~_{b}\nabla^{-(n-\alpha)}f(t) =\frac{(-1)^n\Delta^n}{\Gamma(n-\alpha)} \sum_{s=t}^{b-1}(s-\rho(t))^{\overline{n-\alpha-1}}f(s),~~t \in ~ _{b-1}\mathbb{N}
\end{equation}

\end{defn}

Regarding the domains of the fractional type differences we observe:

(i) The delta left fractional difference $\Delta_a^\alpha$ maps functions defined on $\mathbb{N}_a$ to functions defined on $\mathbb{N}_{a+(n-\alpha)}$.

(ii) The delta right fractional difference $~_{b}\Delta^\alpha$ maps functions defined on $~_{b}\mathbb{N}$ to functions defined on $~_{b-(n-\alpha)}\mathbb{N}$.

(iii) The nabla left fractional difference $\nabla_a^\alpha$ maps functions defined on $\mathbb{N}_a$ to functions defined on $\mathbb{N}_{a+n}$ .

(iv)  The nabla right fractional difference $~_{b}\nabla^\alpha$ maps functions defined on $~_{b}\mathbb{N}$ to functions defined on $~_{b-n}\mathbb{N}$ .

\begin{defn}\label{cd}
Let  $\alpha>0,~\alpha \notin \mathbb{N}$. Then,

(i)\cite{Th Caputo} the delta $\alpha-$order Caputo left  fractional difference of a function $f$ defined on $\mathbb{N}_a$  is defined by
\begin{equation} \label{rd}
~^{C}\Delta_a^\alpha f(t)\triangleq\Delta_a ^{-(n-\alpha)}\Delta
^nf(t)=\frac{1}{\Gamma(n-\alpha)}
\sum_{s=a}^{t-(n-\alpha)}(t-\sigma(s))^{(n-\alpha-1)}\Delta_s^nf(s)
\end{equation}

(ii) \cite{Th Caputo} the delta $\alpha-$ order Caputo right  fractional difference of a function $f$ defined on $~_{b}\mathbb{N}$  is defined by

\begin{equation} \label{ld}
~^{C}_{b}\Delta^\alpha f(t)\triangleq ~_{b}\Delta ^{-(n-\alpha)}\nabla_{\ominus}^nf(t)=\frac{1}{\Gamma(n-\alpha)}
\sum_{s=t+(n-\alpha)}^b(s-\sigma(t))^{(n-\alpha-1)}\nabla_{\ominus}^nf(s)
\end{equation}
where $n=[\alpha]+1$.

If $\alpha =n\in \mathbb{N}$, then
$$~^{C}\Delta_a^\alpha f(t)\triangleq \Delta^n f(t)~~\texttt{and}~
~^{C}_{b}\Delta^\alpha f(t)\triangleq \nabla_b^n f(t)$$
\end{defn}

It is clear that $~^{C}\Delta_a^\alpha$ maps functions defined on
$\mathbb{N}_a$ to functions defined on $\mathbb{N}_{a+(n-\alpha)}$, and that
$~^{C}_{b}\Delta^\alpha$ maps functions defined on $_{b}\mathbb{N}$ to functions
defined on $_{b-(n-\alpha)}\mathbb{N}$.

\begin{thm} \label{relate} \cite{Th Caputo}
For any $\alpha>0$, we have
\begin{equation}\label{relate1}
~^{C}\Delta_a^\alpha f(t)=\Delta_a^\alpha f(t)-\sum_{k=0}^{n-1}
\frac{(t-a)^{(k-\alpha)}}{\Gamma(k-\alpha+1)}\Delta^k f(a)
\end{equation}
and
\begin{equation}\label{relate2}
~_{b}^{C}\Delta^\alpha f(t)=~_{b}\Delta^\alpha f(t)-\sum_{k=0}^{n-1}
\frac{(b-t)^{(k-\alpha)}}{\Gamma(k-\alpha+1)}\nabla_{\ominus}^k f(b).
\end{equation}
In particular, when $0<\alpha<1$, we have
\begin{equation}\label{relate22}
~^{C}\Delta_a f(t)=\Delta_a^\alpha f(t)-
\frac{(t-a)^{(-\alpha)}}{\Gamma(1-\alpha)} f(a).
\end{equation}

\begin{equation}\label{relate4}
~_{b}^{C}\Delta f(t)=~_{b}\Delta^\alpha f(t)-
\frac{(b-t)^{(-\alpha)}}{\Gamma(1-\alpha)} f(b)
\end{equation}
\end{thm}

\begin{defn}\cite{Thdiscrete}
Let $f:\mathbb{N}_a\rightarrow \mathbb{R}$ ($f:~_{b}\mathbb{N}\rightarrow \mathbb{R}$, respectively),  $\alpha >0,~ n=[\alpha]+1,~a(\alpha)=a+n-1$ and $b(\alpha)=b-n+1$. Then the (dual) nabla left and right Caputo fractional differences are defined by
\begin{equation}\label{Cdual left}
   ~^{C}\nabla_{a(\alpha)}^\alpha f(t)=\nabla_{a(\alpha)}^{-(n-\alpha)} \nabla^n f(t),~~t \in \mathbb{N}_{a+n}
\end{equation}

and
\begin{equation}\label{Cdual right}
 ~ _{b(\alpha)} ^{C}\nabla^\alpha f(t)=~_{b(\alpha)}\nabla^{-(n-\alpha)} {\ominus}\Delta^n f(t), ~~t \in ~_{b-n}\mathbb{N},
\end{equation}
respectively.
\end{defn}

The following proposition states a dual relation between left delta Caputo fractional differences and left nabla (dual) Caputo fractional differences.

\begin{prop} \label{lCdual} \cite{Thdiscrete}
For $f:\mathbb{N}_a\rightarrow \mathbb{R}$, $\alpha >0,~ n=[\alpha]+1,~a(\alpha)=a+n-1$, we have

\begin{equation}\label{lCdual one}
   (~^{C} \Delta_a^\alpha f)(t-\alpha)=  (~^{C} \nabla_{a(\alpha)}^\alpha f)(t),~~t \in N_{a+n}.
\end{equation}
\end{prop}

Analogously, the following proposition relates right delta Caputo fractional differences and  right nabla (dual) Caputo fractional differences.
\begin{prop} \label{rCdual} \cite{Thdiscrete}
For $f:~_{b}\mathbb{N}\rightarrow \mathbb{R}$,  $\alpha >0,~ n=[\alpha]+1,~b(\alpha)=b-n+1$, we have

\begin{equation}\label{rCdual one}
   (~^{C} _{b}\Delta^\alpha f)(t+\alpha)=  (~^{C} _{b(\alpha)}\nabla^\alpha f)(t),~~t \in~ _{b-n}\mathbb{N}.
\end{equation}
\end{prop}

 \begin{thm} \label{nabla relate} \cite{Thdiscrete}
For any $\alpha>0$ and $f:\mathbb{N}_a\rightarrow \mathbb{R}$, we have
\begin{equation}\label{nrelate1}
~^{C}\nabla_{a(\alpha)}^\alpha f(t)=\nabla_{a(\alpha)}^\alpha f(t)-\sum_{k=0}^{n-1}
\frac{(t-a(\alpha))^{\overline{k-\alpha}}}{\Gamma(k-\alpha+1)}\nabla^k f({a(\alpha)})
\end{equation}
and
\begin{equation}\label{nrelate2}
~_{b(\alpha)}^{C}\nabla^\alpha f(t)=~_{b(\alpha)}\nabla^\alpha f(t)-\sum_{k=0}^{n-1}
\frac{(b(\alpha)-t)^{\overline{k-\alpha}}}{\Gamma(k-\alpha+1)}~_{\ominus}\Delta^k f(b(\alpha)).
\end{equation}
In particular, when $0<\alpha<1$, then $a(\alpha)=a$ and $b(\alpha) =b$ and hence we have
\begin{equation}\label{nrelate22}
~^{C}\nabla_a^\alpha f(t)=\nabla_a^\alpha f(t)-
\frac{(t-a)^{\overline{-\alpha}}}{\Gamma(1-\alpha)} f(a)
\end{equation}
and
\begin{equation}\label{nrelate4}
~_{b}^{C}\nabla^\alpha f(t)=~_{b}\nabla^\alpha f(t)-
\frac{(b-t)^{\overline{-\alpha}}}{\Gamma(1-\alpha)} f(b)
\end{equation}
\end{thm}

 \indent

 \begin{prop} \label{dnabla trans}\cite{Thdiscrete}
Assume $\alpha>0$ and $f$ is defined on suitable domains  $\mathbb{N}_a$
and $_{b}\mathbb{N}$. Then
\begin{equation}\label{dntrans1}
\nabla_{a(\alpha)}^{-\alpha} ~^{C}\nabla_{a(\alpha)}^\alpha
f(t)=f(t)-\sum_{k=0}^{n-1}\frac{(t-a(\alpha))^{\overline{k}}}{k!}\nabla^kf(a(\alpha))
\end{equation}
and
\begin{equation}\label{dntrans2}
~_{b(\alpha)}\nabla^{-\alpha} ~_{b(\alpha)}^{C}\nabla^\alpha
f(t)=f(t)-\sum_{k=0}^{n-1}\frac{(b(\alpha)-t)^{\overline{k}}}{k!}~_{\ominus}\Delta^kf(b(\alpha)).
\end{equation}
In particular, if $0<\alpha\leq1$ then $a(\alpha)=a$   and $b(\alpha)=b$ and hence
\begin{equation}\label{dntrans3}
\nabla_a^{-\alpha} ~^{C}\nabla_a^\alpha f(t)= f(t)-f(a)~~\texttt{and}~~
~_{b}\nabla^{-\alpha} ~_{b}^{C}\nabla^\alpha f(t)=f(t)-f(b)
\end{equation}

\end{prop}

\begin{lem} \label{left dual} (see \cite{Ferd3} and Lemma 5 in \cite{bino})
Let $0\leq n-1< \alpha \leq n$ and let $y(t)$ be defined on $\mathbb{N}_a$. Then the following statements are valid.

(i)$ (\Delta_a^\alpha) y(t-\alpha)= \nabla_{a-1}^\alpha y(t)$ for $t \in \mathbb{N}_{n+a}$.

(ii) $ (\Delta_a^{-\alpha}) y(t+\alpha)= \nabla_{a-1}^{-\alpha} y(t)$ for $t \in \mathbb{N}_a$.
\end{lem}

Next lemma for the right fractional sums and differences case.

\begin{lem} \label{right dual} \cite{Thsh}
Let $y(t)$ be defined on $~_{b+1}\mathbb{N}$. Then the following statements are valid.

(i)$ (~_{b}\Delta^\alpha) y(t+\alpha)= ~_{b+1}\nabla^\alpha y(t)$ for $t \in ~_{b-n}\mathbb{N}$.

(ii) $ (~_{b}\Delta^{-\alpha}) y(t-\alpha)= ~_{b+1}\nabla^{-\alpha} y(t)$ for $t \in ~_{b}\mathbb{N}$.
\end{lem}

In \cite{Holm} the author used a delta Leibniz's Rule to obtain the following alternative definition for Riemann delta left fractional differences:

\begin{equation} \label{a}
    \Delta_a^\alpha f(t)=\frac{1}{\Gamma(-\alpha)} \sum_{s=a}^{t+\alpha} (t-\sigma(s))^{(-\alpha-1)}f(s),~~\alpha \notin \mathbb{N}, ~~t \in \mathbb{N}_{a+n-\alpha},
\end{equation}

In analogous to (\ref{a}) the authors in \cite{Ahrendt} used a nabla Leibniz's Rule to prove that

\begin{equation}\label{aa}
    \nabla_a^\alpha f(t)= \frac{1}{\Gamma(-\alpha)} \sum_{s=a+1}^t (t-\rho(s))^{\overline{-\alpha-1}}f(s),~~t \in \mathbb{N}_{a+1}\supseteq  \mathbb{N}_{a+n}.
\end{equation}

In \cite{THFer} the authors used a delta Leibniz's Rule to prove  the following formula for nabla right fractional differences
\begin{equation}\label{b}
    ~_{b}\nabla^\alpha f(t)= \frac{1}{\Gamma(-\alpha)} \sum_{s=t}^{b-1} (s-\rho(t))^{\overline{-\alpha-1}}f(s),~~t \in ~_{b-1}\mathbb{N}\supseteq ~_{b-n}\mathbb{N}.
\end{equation}
Similarly, we can use a nabla Leibniz's Rule to prove the following formula for the delta right fractional differences:
\begin{equation}\label{bb}
    ~_{b}\Delta^\alpha f(t)= \frac{1}{\Gamma(-\alpha)} \sum_{s=t-\alpha}^{b} (s-\sigma(t))^{(-\alpha-1)}f(s), ~~t \in ~_{b-(n-\alpha)}\mathbb{N}.
\end{equation}

If $f(s)$ is defined on $N_a\cap ~_{b}N$ and $a\equiv b~ (mod
~1)$ then $(Qf)(s)=f(a+b-s)$. The Q-operator generates a dual identity by which the left type and the right type fractional sums and differences are related. Using the change of variable $u=a+b-s$, in \cite{Th Caputo} it was shown  that
\begin{equation}\label{sum pr}
    \Delta_a^{-\alpha}Qf(t)= Q~_{b}\Delta^{-\alpha}f(t),
\end{equation}

and hence
\begin{equation}\label{pr}
    \Delta_a^\alpha Qf(t)= (Q ~_{b}\Delta^\alpha f)(t).
\end{equation}

and
\begin{equation}\label{cpr}
    ~^{C}\Delta_a^\alpha Qf(t)= (Q ~_{b}~^{C}\Delta^\alpha f)(t).
\end{equation}

The proofs of (\ref{pr}) and (\ref{cpr}) follow by the definition, (\ref{sum pr}) and by noting that

$$-Q\nabla f(t)=\Delta Qf(t).$$

Similarly, in the nabla case we have
\begin{equation}\label{npr sum}
   \nabla_a^{-\alpha}Qf(t)= Q~_{b}\nabla^{-\alpha}f(t),
 \end{equation}

and hence

\begin{equation}\label{rpr}
   \nabla_a^\alpha Qf(t)=( Q ~_{b}\nabla^\alpha f)(t).
\end{equation}

and
\begin{equation}\label{crpr}
   ~^{C}\nabla_a^\alpha Qf(t)=( Q ~_{b}~^{C}\nabla^\alpha f)(t).
\end{equation}

The proofs of (\ref{rpr}) and (\ref{crpr}) follow by the definition, (\ref{npr sum}) and that

$$-Q\Delta f(t)=\nabla Qf(t).$$
For more details about the discrete version of the Q-operator we refer to \cite{Thsh}.

\begin{defn} \cite{Uyanik}
Let $\mathbb{N}_0\rightarrow \mathbb{R}$ be a function satisfying $y(0)\geq 0$. $y$ is called $\nu-$increasing ($\nu-$decreasing) on $\mathbb{N}_0$ if $y(a+1)\geq \nu y(a)$ for all $a \in \mathbb{N}_0$ ($y(a+1)\leq \nu y(a)$ for all $a \in \mathbb{N}_0$ ).
\end{defn}
\begin{thm} \cite{Uyanik}
Let $\mathbb{N}_0\rightarrow \mathbb{R}$ be a function satisfying $y(0)\geq 0$. Fix $\nu \in (0,1)$ and suppose that $\Delta_0^{\nu} y(t)\geq 0$ for each $t \in \mathbb{N}_{1-\nu}$. Then, $y$ is $\nu-$incrasing.
\end{thm}

\section{Monotonicity known results via dual identities}
The following two monotonicity results have been proved in \cite{Jia} for delta and nabla fractional differences separately in two long proofs.

\begin{thm} \cite{Dahal, Jia} \label{JEP1}
If $f:\mathbb{N}_a \rightarrow \mathbb{R}$, $\Delta_a^{\nu}f \geq 0$ for $t \in \mathbb{N}_{a+2-\nu}$ with $1< \nu <2$, and $f(a+1)\geq f(a)\geq 0$, then $\Delta f(t)\geq 0$ for $t \in \mathbb{N}_{a}$ . That is $f$ is nondecreasing on $\mathbb{N}_{a}$.
\end{thm}

\begin{thm}\label{JEP}
If $f:\mathbb{N}_a\rightarrow \mathbb{R}$, $\nabla_a^\nu f(t) \geq 0,$ for each $t \in \mathbb{N}_{a+1}$, with $1< \nu <2$, then $\nabla f(t) \geq 0$ for $\mathbb{N}_{a+1}$. That is $f$ is nondecreasing on $\mathbb{N}_{a}$.
\end{thm}

Assuming Theorem \ref{JEP1} is given, we will use its conclusion together with dual identity in Lemma (\ref{left dual})(a) to re-obtain and confirm Theorem \ref{JEP}. Actually, we state and  prove the following version of Theorem \ref{JEP} with $a$ replaced by $a-1$.

\begin{thm}\label{JEPp}
If $f:\mathbb{N}_{a-1}\rightarrow \mathbb{R}$, $\nabla_{a-1}^\nu f(t) \geq 0,$ for each $t \in \mathbb{N}_{a}$, with $1< \nu <2$, then $\nabla f(t) \geq 0$ for $t \in \mathbb{N}_{a+1}$. That is $f$ is nondecreasing on $\mathbb{N}_{a+1}$.
\end{thm}

\begin{proof}
From the assumption and the representation (\ref{a}) with $a$ replaced by $a-1$, we have
\begin{equation}\label{j1}
  \nabla_{a-1}^\nu f(a)=f(a)\geq 0,
\end{equation}

and

\begin{equation}\label{j2}
  \nabla_{a-1}^\nu f(a+1)=-\nu f(a)+f(a+1) \geq 0.
\end{equation}
Hence, (\ref{j1}) and (\ref{j2}) imply that $f(a+1)\geq \nu f(a) \geq f(a)\geq 0$.

On the other hand, the dual identity Lemma \ref{left dual} (i) implies $\Delta_a^\nu f(t-\nu) =\nabla_{a-1}^\nu f(t)\geq 0$ for all $t \in \mathbb{N}_{a+2}$, or $\Delta_a^\nu f(u)\geq 0$ for all $u \in \mathbb{N}_{a+2-\nu}$. Then, by Theorem \ref{JEP1}, we conclude that $\Delta f(t)=\nabla f(t+1) \geq 0$ for all $t \in \mathbb{N}_a$ or $\nabla f(u) \geq 0$ for all $u \in \mathbb{N}_{a+1}.$
\end{proof}

The following three theorems have been proved in \cite{Slov} very recently for the delta fractional difference operator. We shall use them to prove correspondent nabla ones by making use of the dual identities.

\begin{thm} \label{Slov1}\cite{Slov}
Assume that $f:\mathbb{N}_a\rightarrow \mathbb{R}$ and $\Delta_a^\nu f(t)\geq 0$, for each $t \in \mathbb{N}_{a+2-\nu}$ with $1<\nu <2$. If $f(a+1)\geq \frac{\nu}{k+1}f(a)$  for each $k \in \mathbb{N}_0$, then $\Delta f(t)\geq 0$ for all $t \in \mathbb{N}_{a+1}$.
\end{thm}

Next, we state and prove its nabla version.

\begin{thm} \label{Slov11}
Assume that $f:\mathbb{N}_{a-1}\rightarrow \mathbb{R}$ and $\nabla_{a-1}^\nu f(t)\geq 0$, for each $t \in \mathbb{N}_{a+2}$  with $1<\nu <2$. If $f(a+1)\geq \frac{\nu}{k+2}f(a)$  for each $k \in \mathbb{N}_0$, then $\nabla f(t)\geq 0$ for all $t \in \mathbb{N}_{a+2}$.
\end{thm}

\begin{proof}
By assumption and the dual identity  Lemma \ref{left dual}(i), we have $\Delta_a^\nu f(t-\nu)=\nabla_{a-1}^\nu f(t)\geq 0$ for all $t \in \mathbb{N}_{a+2}$. That is $\Delta_a^\nu f(t)\geq 0$ for each $t \in \mathbb{N}_{a+2-\nu}$. Then, Theorem \ref{Slov1} implies that $\Delta f(t)= \nabla f(t+1)\geq 0$ for all $t \in \mathbb{N}_{a+1}$, or $\nabla f(u)\geq 0$ for all $ u \in \mathbb{N}_{a+2}$.
\end{proof}

\begin{thm} \label{slov2} \cite{Slov}
Assume that $f:\mathbb{N}_a\rightarrow \mathbb{R}$ and $\Delta_a^\nu f(t)\geq 0$, for each $t \in \mathbb{N}_{a+2-\nu}$ with $1<\nu <2$. If $f(a+2)\geq \frac{\nu}{k+2}f(a+1)+\frac{(k+1-\nu)\nu} {(k+2)(k+3)}f(a)$  for each $k \in \mathbb{N}_1$, then $\Delta f(t)\geq 0$ for all $t \in \mathbb{N}_{a+2}$.
\end{thm}

Its nabla correspondent result will be.

\begin{thm} \label{slov22}
Assume that $f:\mathbb{N}_{a-1}\rightarrow \mathbb{R}$ and $\nabla_{a-1}^\nu f(t)\geq 0$, for each $t \in \mathbb{N}_{a+3}$ with $1<\nu <2$. If $f(a+2)\geq \frac{\nu}{k+2}f(a+1)+\frac{(k+1-\nu)\nu}{(k+2)(k+3)}f(a)$  for each $k \in \mathbb{N}_1$, then $\nabla f(t)\geq 0$ for all $t \in \mathbb{N}_{a+3}$.
\end{thm}
We omit the proof since it is similar to above.

\begin{thm} \label{slov3} \cite{Slov}
Assume that $f:\mathbb{N}_a\rightarrow \mathbb{R}$ and $\Delta_a^\nu f(t)\geq 0$, for each $t \in \mathbb{N}_{a+2-\nu}$ with $1<\nu <2$. If $f(a+3)\geq \frac{\nu}{k}f(a+2)+\frac{(k-\nu)\nu}{(k)(k+1)}f(a+1)+\frac{(k+1-\nu)(k-\nu)\nu} {(k+1)(k+2)k}f(a)$  for each $k \in \mathbb{N}_2$, then $\Delta f(t)\geq 0$ for all $t \in \mathbb{N}_{a+3}$.
\end{thm}

The  nabla correspondent of Theorem \ref{slov3} will be:
\begin{thm} \label{slov33}
Assume that $f:\mathbb{N}_{a-1}\rightarrow \mathbb{R}$ and $\nabla_{a-1}^\nu f(t)\geq 0$, for each $t \in \mathbb{N}_{a+4}$ with $1<\nu <2$. If $f(a+3)\geq \frac{\nu}{k}f(a+2)+\frac{(k-\nu)\nu}{k(k+1)}f(a+1)+\frac{(k+1-\nu)(k-\nu)}\nu{(k+1)(k+2)k} f(a)$  for each $k \in \mathbb{N}_2$, then $\nabla f(t)\geq 0$ for all $t \in \mathbb{N}_{a+4}$.
\end{thm}
In \cite{Uyanik}, the following monotonicity result was proved for the delta fractional difference operator, with order $0<\alpha<1$:

\begin{thm} \cite{Uyanik} \label{U1}
Let $f:\mathbb{N}_0\rightarrow \mathbb{R}$ be a function satisfying $y(0)\geq 0$. Fix $\nu \in (0,1)$ and suppose that
\begin{equation}\label{U11}
  \Delta_0^\nu y(t)\geq 0,~~~\texttt{for each}~t \in \mathbb{N}_{1-\nu}.
\end{equation}
Then, $y$ is $\nu-$increasing on $\mathbb{N}_0$.
\end{thm}
By means of the dual identity Lemma \ref{left dual}(i) we can have the following nabla version of Theorem \ref{U1} abo
\begin{thm}  \label{UU1}
Let $f:\mathbb{N}_{0}\rightarrow \mathbb{R}$ be a function . Fix $\nu \in (0,1)$ and suppose that
\begin{equation}\label{UU11}
  \nabla_{-1}^\nu y(t)\geq 0,~~~\texttt{for each}~t \in \mathbb{N}_{0}.
\end{equation}
Then, $y$ is $\nu-$increasing on $\mathbb{N}_0$.
\end{thm}
\begin{proof}
From assumption, $\nabla_{-1}^\nu f(0)=f(0)\geq 0$. On the other hand, $\Delta_0^\nu f(t-\nu)=\nabla_{-1}^\nu f(t)\geq 0$ for $t \in \mathbb{N}_1$. That is $\Delta_0^\nu f(u)\geq 0$ for $t \in \mathbb{N}_{1-\nu}$,  Thus, by Theorem \ref{U1} we conclude that $f$ is $\nu-$increasing on $\mathbb{N}_0$.
\end{proof}
On the other way back, we can similarly use the dual identity and delta version Theorem 3.6 in \cite{Uyanik} to prove the following nabla theorem:

\begin{thm}  \label{UU2}
Let $f:\mathbb{N}_{0}\rightarrow \mathbb{R}$ be a function . Fix $\nu \in (0,1)$ and suppose that $f$ is increasing on $\mathbb{N}_0$ and $y(0)\geq 0$. Then
\begin{equation}\label{UU22}
  \nabla_{-1}^\nu y(t)\geq 0,~~~\texttt{for each}~t \in \mathbb{N}_{0}.
\end{equation}

\end{thm}
\section{Monotonicity results  for Caputo fractional diferences }

\begin{thm}\label{C1delta}
If $f:\mathbb{N}_a\rightarrow \mathbb{R}$ is a function, $~^{C}\Delta_a^\nu f(t)\geq -\frac{(t-a)^{(-\nu)}}{\Gamma(1-\nu)}f(a)
-\frac{(t-a)^{(1-\nu)}}{\Gamma(2-\nu)}\Delta f(a)$ for $t \in \mathbb{N}_{a+2- \nu}$ with $1<\nu <2$ and $f(a+1)\geq f(a)\geq 0$. Then, $\Delta f(t)\geq 0$ for $t \in \mathbb{N}_a$.
\end{thm}
The proof follows by (\ref{relate1}) with $n=2$ and Theorem \ref{JEP1}.

The following is the Caputo version of Theorem \ref{Slov1}.
\begin{thm} \label{C2delta}
Assume that $f:\mathbb{N}_a\rightarrow \mathbb{R}$ and $~^{C}\Delta_a^\nu f(t)\geq -\frac{(t-a)^{(-\nu)}}{\Gamma(1-\nu)}f(a)
-\frac{(t-a)^{(1-\nu)}}{\Gamma(2-\nu)}\Delta f(a)$, for each $t \in \mathbb{N}_{a+2-\nu}$ with $1<\nu <2$. If $f(a+1)\geq \frac{\nu}{k+1}f(a)$  for each $k \in \mathbb{N}_0$, then $\Delta f(t)\geq 0$ for all $t \in \mathbb{N}_{a+1}$.
\end{thm}

The proof follows by (\ref{relate1}) with $n=2$ and Theorem \ref{Slov1}.

\begin{thm} \label{C3delta}
Assume that $f:\mathbb{N}_a\rightarrow \mathbb{R}$ and $~^{C}\Delta_a^\nu f(t)\geq -\frac{(t-a)^{(-\nu)}}{\Gamma(1-\nu)}f(a)
-\frac{(t-a)^{(1-\nu)}}{\Gamma(2-\nu)}\Delta f(a)$ for each $t \in \mathbb{N}_{a+2-\nu}$ with $1<\nu <2$. If $f(a+2)\geq \frac{\nu}{k+2}f(a+1)+\frac{(k+1-\nu)\nu} {(k+2)(k+3)}f(a)$  for each $k \in \mathbb{N}_1$, then $\Delta f(t)\geq 0$ for all $t \in \mathbb{N}_{a+2}$.
\end{thm}
The proof follows by (\ref{relate1}) with $n=2$ and Theorem \ref{slov2}.

\begin{thm} \label{C4delta}
Assume that $f:\mathbb{N}_a\rightarrow \mathbb{R}$ and $~^{C}\Delta_a^\nu f(t)\geq -\frac{(t-a)^{(-\nu)}}{\Gamma(1-\nu)}f(a)
-\frac{(t-a)^{(1-\nu)}}{\Gamma(2-\nu)}\Delta f(a)$, for each $t \in \mathbb{N}_{a+2-\nu}$ with $1<\nu <2$. If $f(a+3)\geq \frac{\nu}{k}f(a+2)+\frac{(k-\nu)\nu}{(k)(k+1)}f(a+1)+\frac{(k+1-\nu)(k-\nu)\nu} {(k+1)(k+2)k}f(a)$  for each $k \in \mathbb{N}_2$, then $\Delta f(t)\geq 0$ for all $t \in \mathbb{N}_{a+3}$.
\end{thm}
The proof follows by (\ref{relate1}) with $n=2$ and Theorem \ref{slov3}.

\begin{thm}  \label{C5delta}
Let $f:\mathbb{N}_0\rightarrow \mathbb{R}$ be a function satisfying $f(0)\geq 0$. Fix $\nu \in (0,1)$ and suppose that
\begin{equation}\label{C11}
  ~^{C}\Delta_0^\nu f(t)\geq -\frac{t^{(-\nu)}}{\Gamma(1-\nu)}f(0),~~~\texttt{for each}~t \in \mathbb{N}_{1-\nu}.
\end{equation}
Then, $f$ is $\nu-$increasing on $\mathbb{N}_0$.
\end{thm}

The proof follows by (\ref{relate1}) with $n=1$ and Theorem \ref{U1}.

\indent

The following is Theorem 3.6 in \cite{Uyanik}.
\begin{thm}  \label{U3}\cite{Uyanik}
Let $f:\mathbb{N}_{0}\rightarrow \mathbb{R}$ be a function . Fix $\nu \in (0,1)$ and suppose that $f$ is increasing on $\mathbb{N}_0$ and $f(0)\geq 0$. Then
\begin{equation}\label{U33}
  \Delta_{0}^\nu f(t)\geq 0,~~~\texttt{for each}~t \in \mathbb{N}_{1-\nu}.
\end{equation}

\end{thm}
By means of Caputo fractional differences, Theorem \ref{U3} takes the form:
\begin{thm}  \label{C6delta}
Let $f:\mathbb{N}_{0}\rightarrow \mathbb{R}$ be a function . Fix $\nu \in (0,1)$ and suppose that $f$ is increasing on $\mathbb{N}_0$ and $f(0)\geq 0$. Then
\begin{equation}\label{U333}
  ~^{C}\Delta_{0}^\nu f(t)\geq -\frac{t^{(-\nu)}}{\Gamma(1-\nu)}f(0) ,~~~\texttt{for each}~t \in \mathbb{N}_{1-\nu}.
\end{equation}
The proof follows by (\ref{relate1}) with $n=1$ and Theorem \ref{U3}.
\begin{rem}
We can prove the results of this section for nabla left Caputo fractional difference operators either by using the correspondence result for the delta left  Caputo fractional differences obtained in this section via the dual identity (\ref{lCdual one}) (for $1<\alpha <2$, $a(\alpha)=a+1$ and for $0<\alpha <1$, $a(\alpha)=a$) or by using the nabla Reimann fractional difference results proved in the previous section via the relation (\ref{nrelate1}).
\end{rem}
\end{thm}
\section{Monotonicity results for right fractional difference types}

In this section, we use the monotonicity results for left fractional difference types discussed in the previous two sections and $Q-$operator dual identities (\ref{pr}), (\ref{rpr}) for delta and nabla Riemann fractional differences, and (\ref{cpr}), (\ref{crpr}) for delta and nabla Caputo fractional differences, to obtain monotonicity results for the right fractional difference types.

\begin{thm} \label{d1}
Assume $f:\mathbb{N}_a \cap~ _{b}\mathbb{N}\rightarrow \mathbb{R}$ is a function, $~_{b}\Delta^{\alpha}f(u)\geq 0$ for $u \in ~_{b-(2-\alpha)}\mathbb{N}$, with $1<\alpha <2$, such that $f(b-1)\geq f(b)\geq 0$. Then, $-\nabla f(t)\geq 0$ for all $t \in ~_{b}\mathbb{N}$.
\end{thm}

\begin{proof}
From the dual identity (\ref{pr}) we have $(\Delta_a^\alpha g)(t)=(Q~_{b}\Delta^\alpha f)(t)$ for all $t \in \mathbb{N}_{a+2-\alpha}$, where $g(t)=f(a+b-t)$. The assumption, $~_{b}\Delta^{\alpha}f(u)\geq 0$ for $u \in ~_{b-(2-\alpha)}\mathbb{N}$ implies that
$\Delta_a^\alpha g(t)\geq 0$ for all $t \in \mathbb{N}_{a+2-\alpha}$. On the other hand, the assumption $f(b-1)\geq f(b)\geq0$ implies that $g(a+1)\geq g(a)\geq 0$. Therefore, Theorem \ref{JEP1} applied to $g(t)=f(a+b-t)$ implies that $\Delta g(t)\geq 0$ for all $t \in \mathbb{N}_a$. Which means that $-\nabla f(t)\geq 0$ for all $t \in ~_{b}\mathbb{N}$.
\end{proof}
The nabla version of Theorem \ref{d1} is then,
\begin{thm} \label{n1}
Assume $f:~_{b+1}\mathbb{N}\rightarrow \mathbb{R}$ is a function, $~_{b+1}\nabla^{\alpha}f(u)\geq 0$ for $u \in ~_{b}\mathbb{N}$, with $1<\alpha <2$. Then, $-\Delta f(t)\geq 0$ for all $t \in ~_{b-1}\mathbb{N}$.
\end{thm}

\begin{proof}
From the dual identity (\ref{rpr}) we have $(\nabla_{a-1}^\alpha g)(t)=(Q~_{b+1}\nabla^\alpha f)(t)$ for all $t \in \mathbb{N}_{a}$, where $g(t)=f(a+b-t)$. The assumption, $~_{b+1}\nabla^{\alpha}f(u)\geq 0$ for $u \in ~_{b}\mathbb{N}$ implies that
$\nabla_{a-1}^\alpha g(t)\geq 0$ for all $t \in \mathbb{N}_{a}$.  Therefore, Theorem \ref{JEPp} applied to $g(t)=f(a+b-t)$ implies that $\nabla g(t)\geq 0$ for all $t \in \mathbb{N}_{a+1}$. Which means that $-\Delta f(t)\geq 0$ for all $t \in ~_{b-1}\mathbb{N}$.
\end{proof}
Similarly, the proof of the following three theorems follow by the dual identity (\ref{pr}) and Theorem \ref{Slov1}, Theorem \ref{slov2} and Theorem \ref{slov3}, respectively.
\begin{thm} \label{d2}
Assume $f:\mathbb{N}_a \cap ~_{b}\mathbb{N}\rightarrow \mathbb{R}$ is a function, $~_{b}\Delta^{\alpha}f(u)\geq 0$ for $u \in ~_{b-(2-\alpha)}\mathbb{N}$, with $1<\alpha <2$, such that $f(b-1)\geq \frac{\alpha}{k+1}f(b)\geq 0,~~k \in \mathbb{N}_0$. Then, $-\nabla f(t)\geq 0$ for all $t \in ~_{b-1}\mathbb{N}$.
\end{thm}
\begin{thm} \label{d3}
Assume $f:\mathbb{N}_a \cap ~_{b}\mathbb{N}\rightarrow \mathbb{R}$ is a function, $~_{b}\Delta^{\alpha}f(u)\geq 0$ for $u \in ~_{b-(2-\alpha)}\mathbb{N}$, with $1<\alpha <2$, such that $f(b-2)\geq \frac{\alpha}{k+2}f(b-1)+\frac{(k+1-\alpha)\nu}{(k+2)(k+3)}f(b)$  for each $k \in \mathbb{N}_1$,. Then, $-\nabla f(t)\geq 0$ for all $t \in ~_{b-2}\mathbb{N}$.
\end{thm}

\begin{thm} \label{d4}
Assume $f:\mathbb{N}_a \cap ~_{b}\mathbb{N}\rightarrow \mathbb{R}$ is a function, $~_{b}\Delta^{\alpha}f(u)\geq 0$ for $u \in ~_{b-(2-\alpha)}\mathbb{N}$, with $1<\alpha <2$, such that $f(b-3)\geq \frac{\alpha}{k}f(b-2)+\frac{(k-\alpha)\alpha}{(k)(k+1)}f(b-1)+\frac{(k+1-\alpha)(k-\nu)\alpha} {(k+1)(k+2)k}f(b)$  for each $k \in \mathbb{N}_2$,. Then, $-\nabla f(t)\geq 0$ for all $t \in ~_{b-3}\mathbb{N}$.
\end{thm}

The next two theorems treat the case when $0<\alpha<1$ for the delta right fractional difference operator.

\begin{thm} \label{d5}
Assume $f: \mathbb{N}_0 \cap ~_{b}\mathbb{N}\rightarrow \mathbb{R}$ ($a=0$ is taken, $b>0$) be a function, $0<\alpha<1$ and $f(b)\geq 0$. Suppose $~_{b}\Delta^\alpha f(u)\geq 0$ for $_{b-(1-\alpha)}\mathbb{N}$. Then, $f$ is $\alpha-$decreasing on $~_{b}\mathbb{N}$. That is $f(t)\geq \alpha f(t+1)$ for all $t \in  ~_{b}\mathbb{N}$.
\end{thm}
The proof follows by the dual identity (\ref{pr}) and Theorem \ref{U1} applied to $g(t)=f(b-t)$.

Conversely, we can state:
\begin{thm} \label{d6}
Assume $f: \mathbb{N}_0 \cap ~_{b}\mathbb{N}\rightarrow \mathbb{R}$ ($a=0$ is taken, $b>0$) be a function, $0<\alpha<1$ and $f(b)\geq 0$. Assume  $f$ is decreasing on $~_{b}\mathbb{N}$. Then, $~_{b}\Delta^\alpha f(u)\geq 0$ for all  $t \in ~_{b-(1-\alpha)}\mathbb{N}$.
\end{thm}
The proof follows by the dual identity (\ref{pr}) and Theorem \ref{U3} applied to $g(t)=f(b-t)$.

\begin{rem}
\begin{enumerate}
  \item The delta right fractional difference results in this section can be carried to nabla right fractional differences via the identity (\ref{rpr}) as we did in Theorem \ref{n1} by using the results of Section 2.
  \item  The results obtained in this section can be carried to Caputo right fractional differences via the identity (\ref{relate2}) in the delta case and (\ref{nrelate2}) in the nabla case. For example, the Caputo version of Theorem \ref{d1} is

      \begin{thm} \label{cd1}
Assume $f:\mathbb{N}_a \cap~ _{b}\mathbb{N}\rightarrow \mathbb{R}$ is a function, $~_{b}^{C}\Delta^{\alpha}f(u)\geq -\frac{(b-u)^{(-\alpha)}}{\Gamma(1-\alpha)}f(b)+\frac{(b-u)^{(1-\alpha)}}{\Gamma(\alpha)}\nabla f(b)$ for $u \in ~_{b-(2-\alpha)}\mathbb{N}$, with $1<\alpha <2$, such that $f(b-1)\geq f(b)\geq 0$. Then, $-\nabla f(t)\geq 0$ for all $t \in ~_{b}\mathbb{N}$.
\end{thm}
and the Caputo version of Theorem \ref{d5} is

\begin{thm} \label{cd5}
Assume $f: \mathbb{N}_0 \cap ~_{b}\mathbb{N}\rightarrow \mathbb{R}$ ($a=0$ is taken, $b>0$) be a function, $0<\alpha<1$ and $f(b)\geq 0$. Suppose $~_{b}^{C}\Delta^\alpha f(u)\geq -\frac{(b-u)^{(-\alpha)}}{\Gamma(1-\alpha)}f(b)$ for $u \in _{b-(1-\alpha)}\mathbb{N}$. Then, $f$ is $\alpha-$decreasing on $~_{b}\mathbb{N}$. That is $f(t)\geq \alpha f(t+1)$ for all $t \in  ~_{b}\mathbb{N}$.
\end{thm}
\end{enumerate}

\end{rem}

\section{Conclusion}
After, we investigated the monotonicity properties for delta and nabla Riemann and Caputo fractional difference operators, we list the following:

\begin{enumerate}
  \item  Using dual identities is useful in providing short proofs for the nabla case using the delta case and for the right case using the left case.
  \item  The monotonicity properties for Caputo fractional difference operators can be proved by using the properties for the Riemann ones via the relation between them. This is very important, since Caputo fractional differences is different from Riemann fractional differences in many aspects.
  \item  The discrete version of the $Q-$operators plays an important role in proving the monotonicity properties for the right case given the result for the left case.
  \item  We noticed that the positivity of the left fractional differences under certain extra starting conditions in the delta case implies that the function is increasing for $1<\alpha <2$ and $\alpha-$increasing for $0<\alpha <1$. While positivity of the right fractional difference implies that the function is decreasing for $1<\alpha <2$ and $\alpha-$decreasing for $0<\alpha <1$. However, we have to be careful that the right case is not equivalent to the case when we assume that the delta fractional difference is negative, since the sign for the starting type condition is different.
  \item  In some cases we need starting conditions when we deal with delta fractional differences, while we may not need this type of condition for the nabla case. The dual identities clarify this.
\end{enumerate}

\end{document}